\documentclass[11pt,a4paper]{article}
\usepackage{amsmath,amssymb,amsthm,enumerate,mathtools,xfrac}
\usepackage{authblk}
\usepackage{relsize}
\usepackage{geometry}
\usepackage{url}

\usepackage{hyperref}
\hypersetup{nesting=true,debug=true,naturalnames=true}
\usepackage{graphicx,upref}

\newtheorem{thm}{Theorem}[section]
\newtheorem{lem}[thm]{Lemma}

\newtheorem{prop}[thm]{Proposition}
\newtheorem{cor}[thm]{Corollary}

\theoremstyle{definition}
\newtheorem{defn}[thm]{Definition}
\newtheorem{ex}[thm]{Example}
\newtheorem{rem}[thm]{Remark}

\newtheorem*{notn}{Notation}

\newcommand{\condb}{(*)}
\newcommand{\condc}{(**)}

\newcommand{\Ob}{\mathrm{Ob}}
\newcommand{\nar}{\unlhd_{\mathrm{nar}}}
\newcommand{\gar}{\succ_{\mathrm{nar}}}

\newcommand{\bN}{\mathbb{N}}

\newcommand{\bZ}{\mathbb{Z}}

\newcommand{\mc}{\mathcal}

\newcommand{\inv}{^{-1}}
\newcommand{\triv}{\{1\}}
\newcommand{\N}{\mathrm{N}}
\newcommand{\OO}{\mathrm{O}}
\newcommand{\CC}{\mathrm{C}}
\newcommand{\Z}{\mathrm{Z}}
\newcommand{\M}{\mathrm{M}}

\newcommand{\defbold}{\textbf}

\begin{document}

\title{Inverse system characterizations of the (hereditarily) just infinite property in profinite groups}

\author{Colin D. Reid}

\maketitle

\begin{abstract}We give criteria on an inverse system of finite groups that ensure the limit is just infinite or hereditarily just infinite.  More significantly, these criteria are `universal' in that all (hereditarily) just infinite profinite groups arise as limits of the specified form.

This is a corrected and revised version of \cite{ReidInvLim}.\end{abstract}

\section{Introduction}

\begin{notn}In this paper, all groups will be profinite groups, all homomorphisms are required to be continuous, and all subgroups are required to be closed; in particular, all references to commutator subgroups are understood to mean the closures of the corresponding abstractly defined subgroups.  For an inverse system
\[
\Lambda = \{(G_n)_{n > 0},\rho_n: G_{n+1} \twoheadrightarrow G_n\}
\]
of finite groups, we require all the homomorphisms $\rho_n$ to be surjective.  A subscript $o$ will be used to indicate open inclusion, for instance $A \leq_o B$ means that $A$ is an open subgroup of $B$.  We use `pronilpotent' and `prosoluble' to mean a group that is the inverse limit of finite nilpotent groups or finite soluble groups respectively, and `$G$-invariant subgroup of $H$' to mean a subgroup of $H$ normalized by $G$.
\end{notn}

A profinite group $G$ is \defbold{just infinite} if it is infinite, and every nontrivial normal subgroup of $G$ is of finite index; it is \defbold{hereditarily just infinite} if in addition every open subgroup of $G$ is just infinite.

At first sight the just infinite property is a qualitative one, like that of simplicity: either a group has nontrivial normal subgroups of infinite index, or it does not.  However, it has been shown by Barnea, Gavioli, Jaikin-Zapirain, Monti and Scoppola (see \cite{BGJMS}, especially Theorem 36) and the present author (\cite{Rei}) that the just infinite property in profinite groups can be characterized by properties of the lattice of \emph{open} normal subgroups, and as such may be regarded as a kind of boundedness property on the finite images (see \cite[Theorems B1 and B2]{Rei}); moreover, it suffices to consider any collection of finite images that form an inverse system for the group.  Similar considerations apply to the hereditarily just infinite property.

The weakness of the characterization given in \cite{Rei} is that the conditions imposed on the finite images are asymptotic: no correspondence is established between the (hereditarily) just infinite property \textit{per se} and the structure of any given finite image.  In a sense this is unavoidable, as any finite group appears as the image of a hereditarily just infinite profinite group (see Example \ref{primex}).  By contrast in \cite{BGJMS}, a strong `periodic' structure of some just infinite pro-$p$ groups is described, but this structure can only exist for a certain special class of just infinite virtually pro-$p$ groups, those which have a property known as finite obliquity (essentially in the sense of \cite{KLP}).  In the present paper we therefore take a new approach, which is to show the existence of an inverse system with certain specified properties for any (hereditarily) just infinite profinite group, and in turn to show that the specified properties imply the (hereditarily) just infinite property in the limit.  The inverse system characterizations in this paper are essentially variations on the following:

\begin{thm}\label{introthm}Let $\Lambda = \{(G_n)_{n > 0},\rho_n: G_{n+1} \rightarrow G_n\}$ be an inverse system of finite groups.  Let $A_n$ be a nontrivial normal subgroup of $G_n$ and let $P_n = \rho_n(A_{n+1})$ for $n > 0$.  Suppose furthermore, for each $n > 0$:
\begin{enumerate}[(i)]  
\item $A_{n+1} > \ker\rho_n \ge P_{n+1}$;
\item $A_n$ has a unique maximal $G_n$-invariant subgroup;
\item Each normal subgroup of $G_n$ contains $P_n$ or is contained in $A_n$ (or both).
\end{enumerate}

Then $G = \varprojlim G_n$ is just infinite.  Conversely, every just infinite profinite group is the limit of such an inverse system.\end{thm}

A similar characterization applies to hereditarily just infinite groups: see Theorem~\ref{mainjithm:hji}.  As shown by Wilson in \cite{Wil}, such groups are not necessarily virtually pro-$p$; we derive one of Wilson's constructions as a special case.

We will also discuss some examples of hereditarily just infinite profinite groups that are not virtually prosoluble, illustrating some features of this class.  In particular, the following serves as a source of examples of such groups:

\begin{defn}\label{def:subprim}Let $X$ be a set and let $G$ be a group acting on $X$, with kernel $K$.  We say $G$ acts \defbold{subprimitively} on $X$ if for every normal subgroup $H$ of $G$, then $H/(H \cap K)$ acts faithfully on every $H$-orbit.\end{defn}

\begin{prop}\label{intro:subprim}Let $G$ be a just infinite profinite group.  Suppose there are infinitely many nonabelian chief factors $R/S$ of $G$ for which $G$ permutes the simple factors of $R/S$ subprimitively.  Then $G$ is hereditarily just infinite.\end{prop}

\begin{rem}The following classes of permutation groups are all special cases of transitive subprimitive permutation groups: regular permutation groups, primitive permutation groups, quasiprimitive groups in the sense of \cite{Praeger} and semiprimitive groups in the sense of \cite{BM}.  An intransitive faithful action is subprimitive if and only if the group acts faithfully and subprimitively on every orbit.\end{rem}

\paragraph{Acknowledgements}
My thanks go to Davide Veronelli, who pointed out some errors in the original article and proposed some corrections; the need to fix the errors in the original article was the immediate motivation for this revised article.  I also thank John Wilson for showing me a preliminary version of \cite{Wil}; Charles Leedham-Green for his helpful comments; and Laurent Bartholdi for his suggestion to look at \cite{Lucchini} and determine whether the group in question is hereditarily just infinite.

\section{Preliminaries}

\begin{defn}Let $G$ be a profinite group.  A \defbold{chief factor} of $G$ is a quotient $K/L$ of some normal subgroup $K$ of $G$, such that $L \unlhd G$ and there are no normal subgroups $M$ of $G$ satisfying $K < M < L$.

We say two chief factors $K_1/L_1$ and $K_2/L_2$ of $G$ are \defbold{associated} to each other if 
\[
K_1L_2 = K_2L_1; \; K_i \cap L_1L_2 = L_i \; (i = 1,2).
\]
This also implies that $K_1L_2=K_1K_2$.

Given a normal subgroup $N$ of $G$, say $N$ \defbold{covers} the chief factor $K/L$ if $NL \ge K$.\end{defn}

The association relation is not transitive in general.  For instance, if $G$ is the Klein $4$-group with subgroups $\{H_1,H_2,H_3\}$ of order $2$, then the pairs of associated chief factors are those of the form $\{G/H_i,H_j/\triv\}$ for $i \neq j$.  Thus $H_1/\triv$ is associated to $G/H_2$ and $G/H_2$ is associated to $H_3/\triv$, but $H_1/\triv$ and $H_3/\triv$ are not associated to each other.

\emph{Nonabelian} chief factors however are much better behaved under association.  A theory of association classes of nonabelian chief factors is developed in a much more general context in \cite{RW_Polish}.

Parts (i) and (ii) of the following lemma are standard facts about profinite groups and will be used without further comment.

\begin{lem}\label{chieflem}Let $G$ be a profinite group.
\begin{enumerate}[(i)]
\item Let $K$ and $L$ be normal subgroups of $G$ such that $K > L$.  Then there exists $L \le M < K$ such that $K/M$ is a chief factor of $G$.
\item Every chief factor of $G$ is finite.
\item Let $K_1/L_1$ and $K_2/L_2$ be associated chief factors of $G$.  Then $K_1/L_1 \cong K_2/L_2$ and $\CC_G(K_1/L_1) = \CC_G(K_2/L_2)$.
\end{enumerate}
\end{lem}

\begin{proof}Given normal subgroups $K$ and $L$ of $G$ such that $K > L$, then $L$ is an intersection of open normal subgroups of $G$, so there is an open normal subgroup $O$ of $G$ containing $L$ but not $K$.  Then $K \cap O \unlhd G$  and $K \cap O <_o K$.  In particular, $K/(K \cap O)$ is a finite normal factor of $G$.  This implies (i) and (ii).

For (iii), let $H = K_1K_2=K_1L_2=K_2L_1$.  Since $K_1 \cap L_1L_2 = L_1$, the kernel of the natural homomorphism from $K_1$ to $H/L_1L_2$ is exactly $L_1$; since $H = K_1L_2$, this homomorphism is also surjective.  Thus $K_1/L_1 \cong H/L_1L_2$.  Let $T$ be a subgroup of $G$.  If $[H,T] \le L_1L_2$ then $[K_1,T] \le L_1L_2 \cap K_1 = L_1$, and conversely if $[K_1,T] \le L_1$ then $[H,T]$ is contained in the normal closure of $L_1 \cup [L_2,T]$ (since $H = K_1L_2$) and thus in $L_1L_2$.  Hence $\CC_G(K_1/L_1) = \CC_G(H/L_1L_2)$.  Similarly $K_2/L_2 \cong H/L_1L_2$ and $\CC_G(K_2/L_2) = \CC_G(H/L_1L_2)$.
\end{proof}

We will also need some definitions and results from \cite{Rei}.

\begin{defn}The \emph{cosocle} or \emph{Mel'nikov subgroup} $\M(G)$ of $G$ is the intersection of all maximal open normal subgroups of $G$. \end{defn}

\begin{lem}[{see \cite[Lemma~2.2]{Rei} and its corollary; see also \cite{Zal}}]\label{melfin} Let $G$ be a just infinite profinite group, and let $H$ be an open subgroup of $G$.  Then $|H:\M(H)|$ is finite.\end{lem}

\begin{defn}Given a profinite group $G$ and subgroup $H$, we define $\Ob_G(H)$ and $\Ob^*_G(H)$ as follows:
\[ \Ob_G(H) := H \cap \bigcap \{K \unlhd_o G \mid K \not\le H\}\]
\[ \Ob^*_G(H) := H \cap \bigcap \{K \leq_o G \mid H \le \N_G(K), \; K \not\le H\}. \]
Note that $\Ob_G(H)$ and $\Ob^*_G(H)$ have finite index in $H$ if and only if the relevant intersections are finite.\end{defn}

\begin{thm}[{\cite[Theorem~36]{BGJMS}, \cite[Theorem~A and Corollary~2.6]{Rei}}]\label{genob} Let $G$ be a just infinite profinite group, and let $H$ be an open subgroup of $G$.  Then $|G:\Ob_G(H)|$ is finite.  If $G$ is hereditarily just infinite, then $|G:\Ob^*_G(H)|$ is finite.\end{thm}

\begin{cor}\label{chiefin}Let $G$ be a just infinite profinite group and let $H$ be an open subgroup of $G$.  Then $K \le H$ for all but finitely many chief factors $K/L$ of $G$.\end{cor}

\begin{proof}By Theorem \ref{genob}, there are only finitely many normal subgroups of $G$ not contained in $H$.  In turn, if $K/L$ is a chief factor of $G$ then $L \ge \M(K)$, so by Lemma \ref{melfin}, the quotients of a given open normal subgroup of $G$ can only produce finitely many chief factors of $G$.\end{proof}

\section{Narrow subgroups}

The key idea in this paper is that of a `narrow' subgroup associated to a chief factor.  These are a general feature of profinite groups, but they have further properties that will be useful in establishing the just infinite property.  Throughout this section, $G$ will be a profinite group.

\begin{defn}Let $1 < A \unlhd G$, and define $\M_G(A)$ to be the intersection of all maximal open $G$-invariant subgroups of $A$.  Note that $\M(A) \le \M_G(A) < A$.  Say $A$ is \emph{narrow} in $G$ and write $A \nar G$ if there is a unique maximal $G$-invariant subgroup of $A$, in other words $\M_G(A)$ is the maximal $G$-invariant subgroup of $A$.

Note that if $A \nar G$, then $A/N \nar G/N$ and $\M_G(A/N) = \M_G(A)/N$ for any $N \unlhd G$ such that $N < A$.

Given $A \nar G$ and a chief factor $K/L$ of $G$, we will say $A$ is associated to $K/L$ to mean $A/\M_G(A)$ is associated to $K/L$.\end{defn}

\begin{lem}\label{critlem}Let $A$ and $K$ be normal subgroups of the profinite group $G$.
\begin{enumerate}[(i)]
\item We have $K\M_G(A) \ge A$ if and only if $K \ge A$.
\item Suppose $A$ is narrow in $G$ and that $KN \ge A$ for some proper $G$-invariant subgroup of $A$.  Then $K \ge A$.
\end{enumerate}
\end{lem}

\begin{proof}
(i)
Suppose $K \ngeq A$.  Then $K \cap A$ is contained in a maximal $G$-invariant subgroup $R$ of $A$.  By the modular law, $K\M_G(A) \cap A = \M_G(A)(K \cap A) \le R < A$, so $K\M_G(A) \ngeq A$.  The converse is clear.

(ii)
Since $A$ is narrow in $G$, the group $\M_G(A)$ contains every proper $G$-invariant subgroup of $A$.  In particular, $\M_G(A) \ge N$, so $K\M_G(A) \ge KN \ge A$ and hence $K \ge A$ by part (i).\end{proof}

The existence of narrow subgroups in profinite groups is shown by a compactness argument.

\begin{lem}\label{narrowassoc}Given any chief factor $K/L$ of $G$, there is a narrow subgroup $A$ of $G$ associated to $K/L$.  Those $A \nar G$ associated to $K/L$ are precisely those narrow subgroups of $G$ contained in $K$ but not $L$, and it follows in this case that $A \cap L = \M_G(A)$.  In particular, every nontrivial normal subgroup of $G$ contains a narrow subgroup of $G$.\end{lem}

\begin{proof}Suppose $A \nar G$ and $A$ is associated to $K/L$.  Then $K\M_G(A) \ge A$, so $K \ge A$ by Lemma \ref{critlem}.  Also $AL \ge K$, so $A \not\le L$.  Conversely, let $A \nar G$ such that $A \le K$ and $A \not\le L$.  Then $AL = K$ since $K/L$ is a chief factor of $G$, and clearly then $AL = K(A \cap L)$ and $(A \cap L)L < AK$.  To show $A \cap L = \M_G(A)$ and that $A/\M_G(A)$ is associated to $K/L$, it remains to show that $A/A \cap L$ is a chief factor of $G$: this is the case as any $G/A \cap L$-invariant subgroup of $A/A \cap L$ would correspond via the isomorphism theorems to a $G/L$-invariant subgroup of $AL/L = K/L$.

It remains to show that narrow subgroups with the specified properties exist.  Let $\mc{K}(G,K)$ be the set of normal subgroups of $G$ contained in $K$ and let $\mc{D} = \mc{K}(G,K) \setminus \mc{K}(G,L)$.  Given that $K \setminus L$ is compact, one sees that the intersection of any descending chain in $\mc{D}$ is not contained in $L$, and is thus an element of $\mc{D}$.  Hence $\mc{D}$ has a minimal element $A$ by Zorn's lemma.  Now any normal subgroup of $G$ properly contained in $A$ must be contained in $L$ by the minimality of $A$ in $\mc{D}$.  Thus $A \cap L$ is the unique maximal $G$-invariant subgroup of $A$ and so $A \nar G$.\end{proof}

We now define a relation on chief factors which gives the intuition underlying the rest of this paper.

\begin{defn}Given chief factors $a = K_1/L_1$ and $b = K_2/L_2$, say $a \gar b$ if $L_1 \ge K_2$ and $\M_{G/L_2}(K_1/L_2) = L_1/L_2$ (in particular, $K_1/L_2 \nar G/L_2$).\end{defn}

\begin{prop}\label{transprop}
\begin{enumerate}[(i)]  
\item Let $K_1/L_1$ and $K_2/L_2$ be chief factors of $G$.  Let $N$ be a normal subgroup of $G$ that covers $K_1/L_1$ and suppose $K_1/L_1 \gar K_2/L_2$.  Then $N$ covers $K_2/L_2$.
\item The relation $\gar$ is a strict partial order.
\item Suppose $K_1/L_1 \gar K_2/L_2 \gar \dots$ is a descending sequence of open chief factors of $G$ such that $\bigcap_i L_i = \triv$.  Then $K_1 \nar G$.\end{enumerate}
\end{prop}

\begin{proof}(i) We have $\M_{G/L_2}(K_1/L_2) = L_1/L_2$, so $NL_2/L_2 \ge K_1/L_2$ by Lemma \ref{critlem}, that is, $NL_2 \ge K_1$.  As $K_2 \le K_1$, this implies that $N$ covers $K_2/L_2$.

(ii) It is clear that $\gar$ is antisymmetric and antireflexive.  Suppose $K_1/L_1 \gar K_2/L_2$ and $K_2/L_2 \gar K_3/L_3$.  It remains to show that $L_1$ is the unique maximal proper $G$-invariant subgroup of $K_1$ containing $L_3$.  Suppose there is another such subgroup $R$.  Since $M_{G/L_2}(K_1/L_2) = L_1/L_2$ and $R$ is not contained in $L_1$, we must have $RL_2 = K_1$.  In particular, $RL_2 \ge K_2$.  Since $L_2/L_3 = \M_{G/L_3}(K_2/L_3)$, we have $R/L_3 \ge K_2/L_3$ by Lemma \ref{critlem}, that is $R \ge K_2$.  But then $L_2 \le R$, so in fact $R = K_1$, a contradiction.

(iii) It follows from (ii) that $K_1/L_i \nar G/L_i$ for all $i$.  Suppose there is a maximal proper $G$-invariant subgroup $N$ of $K_1$ other than $L_1$.  Then $N \not\le L_1$, so $NL_i \not\le L_1$ for all $i$.  Since $K_1/L_i \nar G/L_i$ with $\M_{G/L_i}(K_1/L_i) = L_1/L_i$, this forces $NL_i = K_1$ for all $i$.  It follows by a standard compactness argument that $N = K_1$, a contradiction.\end{proof}

The relation $\gar$ can be used to obtain some restrictions on the just infinite images of a profinite group.

\begin{thm}\label{narchthm}Let $G$ be a profinite group.
\begin{enumerate}[(i)]  
\item Suppose $K_1/L_1 \gar K_2/L_2 \gar \dots$ is a descending sequence of open chief factors of $G$ and let $L = \bigcap_i L_i$.  Then there is a normal subgroup $K\ge L$ of $G$ such that $G/K$ is just infinite and such that for all $i$, $K$ does not cover $K_i/L_i$.  Indeed, it suffices for $K \ge L$ to be maximal subject to not covering the factors $K_i/L_i$.
\item Every just infinite image $G/K$ of $G$ arises in the manner described in (i) with $K=L$.\end{enumerate}
\end{thm}

\begin{proof}(i) Let $\mc{N}$ be the set of all normal subgroups of $G$ which contain $L$ and which do not cover $K_i/L_i$ for any $i$, let $\mc{C}$ be a chain in $\mc{N}$, let $R = \overline{\bigcup \mc{C}}$ and let $i \in \bN$.  Then $RL_i = (\bigcup \mc{C})L_i = CL_i$ for some $C \in \mc{C}$, since $L_i$ is already open and in particular of finite index in $G$; so $RL_i$ does not contain $K_i$, which ensures $R \in \mc{N}$.  Hence $\mc{N}$ has a maximal element $K$ by Zorn's lemma.  Since $K$ does not cover any of the factors $K_i/L_i$, we have $KL_1 > KL_2 > KL_3 > \dots$ and so $|G:K|$ is infinite.  Let $P$ be a normal subgroup of $G$ properly containing $K$.  Then $P$ covers $K_i/L_i$ for some $i$ by the maximality of $K$.  Moreover $M_{G/K}(K_i/L) = L_iL/L$ by Proposition \ref{transprop} and $P \ge L$, so $P \ge K_i$ by Lemma \ref{critlem}.  In particular $|G:P| \le |G:K_i| < \infty$.  Hence $G/K$ is just infinite.

(ii) By Lemma \ref{narrowassoc}, there are narrow subgroups of $G/K$ associated to every chief factor of $G/K$.  Let $R/S$ be a chief factor of $G$ such that $S \ge K$ and let $K_1/K$ be a narrow subgroup of $G/K$ associated to $R/S$, with $\M_{G/K}(K_1/K) = L_1/K$.  Thereafter we choose $K_{i+1}/K$ to be a narrow subgroup associated to a chief factor $R/S$ such that $S \ge K$ and $R \le L_i$.  It is clear that the chief factors $K_i/L_i$ will have the required properties.\end{proof}

\section{Characterizations of the just infinite property and control over chief factors}

Here is a universal inverse limit construction for just infinite profinite groups along the lines of Theorem \ref{introthm}, also incorporating some information about chief factors.

\begin{thm}\label{mainjithm}Let $G$ be a just infinite profinite group.  Let $\mc{C}_1,\mc{C}_2,\dots$ be a sequence of classes of finite groups, such that $G$ has infinitely many chief factors in $\mc{C}_n$ for all $n$.  Then $G$ is the limit of an inverse system $\Lambda = \{(G_n)_{n > 0},\rho_n: G_{n+1} \twoheadrightarrow G_n\}$ as follows:

Each $G_n$ has a specified nontrivial normal subgroup $A_n$ such that, setting $P_n = \rho_{n}(A_{n+1})$:
\begin{enumerate}[(i)]  
\item $\M_{G_{n+1}}(A_{n+1}) \ge \ker\rho_{n} \ge P_{n+1}$;
\item Each normal subgroup of $G_n$ contains $P_n$ or is contained in $A_n$ (or both);
\item $P_n$ is a minimal normal subgroup of $G_n$;
\item $P_n \in \mc{C}_n$ for all $n$.
\end{enumerate}

Conversely, any inverse system satisfying conditions (i) and (ii) (for some choice of nontrivial normal subgroups $A_n$) for all but finitely many $n$ has a limit that is just infinite.\end{thm}

\begin{proof}Suppose that $G$ is just infinite with the specified chief factors.  We will obtain an infinite descending chain $(K_n)_{n \ge 0}$ of narrow subgroups of $G$, and then use these to construct the required inverse system.

Let $K_0 = G$.  Suppose $K_n$ has been chosen, and let $L= \Ob_G(\M_G(K_n))$.  Then by Theorem \ref{genob}, $L$ is open in $G$, and hence by Corollary \ref{chiefin}, all but finitely many chief factors $R/S$ of $G$ satisfy $R \le L$; note also that $L$ is a proper subgroup of $K_n$.  Let $R/S$ be a chief factor such that $R \le L$ and $R/S \in \mc{C}_n$, and let $K_{n+1}$ be a narrow subgroup of $G$ associated to $R/S$.  Then $K_{n+1} \le L$ and $K_{n+1}/\M_G(K_{n+1}) \cong R/S$.  Since $K_{n+1} \le L$, every normal subgroup of $G$ contains $K_{n+1}$ or is contained in $\M_G(K_n)$ (or both).

Set $G_n = G/\M_G(K_{n+1})$, set $P_n = K_{n+1}/\M_G(K_{n+1})$, set $A_n = K_{n}/\M_G(K_{n+1})$ and let the maps $\rho_n$ be the natural quotient maps.  The inverse limit of $(G_n,\rho_n)$ is then an infinite quotient of $G$, which is then equal to $G$ since $G$ is just infinite.  Given a normal subgroup $N/\M_G(K_{n+1})$ of $G_n$ such that $N/\M_G(K_{n+1}) \nleq A_{n}$, then $N \nleq K_{n}$, so $N \ge K_{n+1}$ and hence $N/\M_G(K_{n+1}) \ge P_n$, so condition (ii) is satisfied.  The other conditions are clear.

Now suppose we are given an inverse system satisfying (i) and (ii) for $n \ge n_0>1$, with inverse limit $G$.  Let $\pi_n: G \rightarrow G_n$ be the surjections associated to the inverse limit.  Note that condition (i) ensures that the groups $\pi\inv_n(A_n)$ form a descending chain of open normal subgroups of $G$; since $P_{n+1} \le \ker\rho_n$, we see that $\pi\inv_{n+2}(A_{n+2}) \le \ker\pi_n$, and consequently the subgroups $\pi\inv_n(A_n)$ have trivial intersection.  The fact that $A_{n}$ is nontrivial ensures that $\pi\inv_n(A_n)$ is nontrivial; since the descending chain $(\pi\inv_n(A_n))$ has trivial intersection, $G$ must be infinite.  Let $N$ be a nontrivial normal subgroup of $G$.  Then there is some $n_1 \ge n_0$ such that for all $n \ge n_1$, $N$ is not contained in $\pi\inv(A_n)$.  By (ii) it follows that $\pi_n(N)$ contains $P_n$; hence $\rho_n(\pi_{n+1}(N) \cap A_{n+1})=P_n$ and thus
\[
A_{n+1} = (\pi_{n+1}(N) \cap A_{n+1})\ker\rho_n = (\pi_{n+1}(N) \cap A_{n+1})\M_{G_{n+1}}(A_{n+1}),
\]
using (i).  We conclude by Lemma~\ref{critlem} that $\pi_{n+1}(N) \ge A_{n+1}$.  Hence $\pi_{n+1}(N) \ge \ker\rho_{n}$ for all $n \ge n_1$, so in fact $N \ge \ker\pi_{n_1}$.  In particular, $N$ is open in $G$, showing that $G$ is just infinite.
\end{proof}

\begin{proof}[Proof of Theorem~\ref{introthm}]
Let $G$ be a just infinite profinite group and let $A_n$ and $P_n$ be as obtained in Theorem~\ref{mainjithm}.  Conditions (i) and (iii) of Theorem~\ref{mainjithm} together imply that $A_{n+1} > \M_{G_{n+1}}(A_{n+1}) = \ker\rho_{n}$ and that $\ker\rho_n$ is a maximal $G_{n+1}$-invariant subgroup of $A_{n+1}$, so $A_{n+1}$ is a narrow normal subgroup of $G_{n+1}$, as required for condition (ii) of the present theorem.  The other two conditions are now clear.  Thus $G$ is the limit of an inverse system satisfying the given conditions.

Conversely, let $\Lambda = \{(G_n)_{n > 0},\rho_n: G_{n+1} \twoheadrightarrow G_n\}$ be an inverse system of finite groups with subgroups $A_n$ and $P_n$ of $G_n$ satisfying the given conditions, and let $G = \varprojlim G_n$.  Conditions (i) and (ii) together of the present theorem imply condition (i) of Theorem~\ref{mainjithm}, and condition (iii) of the present theorem is condition (ii) of Theorem~\ref{mainjithm}.  Thus $G$ is just infinite by Theorem~\ref{mainjithm}.
\end{proof}

\begin{rem}
A similar statement was given in the original published version of this article (\cite[Theorem~4.1]{ReidInvLim}).  However, the statement and proof given in \cite{ReidInvLim} fail to take sufficient account of the kernels of the maps $\rho_n$.  We rectify this issue here and in the results derived from Theorem~\ref{mainjithm} by explicitly imposing a condition on $\ker\rho_n$.
\end{rem}

If $G$ is pronilpotent, then all chief factors of $G_n$ are central.  But if $G$ is not virtually pronilpotent, it could be helpful to replace condition (ii) of Theorem \ref{mainjithm} with the stronger but easier-to-verify condition that $A_n$ contains the centralizer of the chief factor $P_n/\M_G(P_n)$.  Indeed, this can always be arranged, with some inevitable adjustments to the classes $\mc{C}_n$.

\begin{defn}Given a profinite group $G$ and a prime $p$, write $\OO^p(G)$ for the intersection of all normal subgroups of $G$ of $p$-power index.\end{defn}

\begin{lem}\label{opsch}Let $G$ be a finite group and let $p$ be a prime.  Suppose that all chief factors of $G$ of exponent $p$ are central, and that $p$ does not divide the order of the Schur multiplier of any nonabelian composition factor of $G$.  Then $\OO^p(G)$ has no composition factors of order $p$.\end{lem}

\begin{proof}It suffices to show $\OO^p(G) < G$ whenever $G$ has a composition factor of order $p$, as then one can repeat the argument for $\OO^p(G)$ and obtain the conclusion from the fact that $\OO^p(\OO^p(G))=\OO^p(G)$.  Let $N$ be a normal subgroup of largest order such that $\OO^p(N) < N$.  Suppose $N < G$, and let $K/N$ be a minimal normal subgroup of $G/N$; note that every $G$-chief factor of $N/\OO^p(N)$ is central in $G$, so $K/\OO^p(N)$ is an iterated central extension of $K/N$.  If $K/N$ is abelian, then $[K,K] \leq N$, so $K/\OO^p(N)$ is nilpotent.  On the other hand, if $K/N$ is nonabelian, then it is a direct power of a nonabelian finite simple group $S$, such that the Schur multiplier of $S$ has order coprime to $p$; it follows that $K/[K,K]\OO^p(N)$ is a nontrivial $p$-group.  In either case $\OO^p(K) < K$, contradicting the choice of $N$.\end{proof}

We now specify a condition $\condb$ on classes of finite groups $\mc{C}$:

$\condb$ The class $\mc{C}$ consists of finite characteristically simple groups.  For each prime $p$, if $\mc{C}$ contains some elementary abelian $p$-group, then within the class of finite groups, $\mc{C}$ contains all elementary abelian $p$-groups and all direct powers of all finite simple groups $S$ such that $p$ divides the Schur multiplier of $S$.

\begin{lem}\label{seclem}Let $G$ be a just infinite profinite group that is not virtually pronilpotent and let $H$ be an open subgroup of $G$.  Let $\mc{C}$ be a class of groups satisfying $\condb$.  Let $\mc{D}$ be the set of chief factors of $G$ belonging to $\mc{C}$, and suppose $\mc{D}$ is infinite.  Then $K\CC_G(K/L) \le H$ for infinitely many $K/L \in \mc{D}$.\end{lem}

\begin{proof}By Corollary \ref{chiefin}, $K \le H$ for all but finitely many chief factors $K/L$ of $G$.  It thus suffices to assume that, for all but finitely many $K/L \in \mc{D}$, $\CC_G(K/L)$ is not contained in $H$, and derive a contradiction.  As $\CC_G(K/L)$ is a normal subgroup of $G$, by Theorem \ref{genob} there are only finitely many possibilities for the subgroups $\CC_G(K/L)$ that are not contained in $H$.  Thus $R = \bigcap_{K/L \in \mc{D}} \CC_G(K/L)$ has finite index in $G$.  Moreover, given $K/L \in \mc{D}$ such that $K \le R$, then $K/L$ is central in $R$ and in particular abelian.  Let $N$ be the smallest $G$-invariant subgroup of $R$ such that $R/N$ has a $G$-chief series whose factors are all in $\mc{D}$ (equivalently, the intersection of all such subgroups).  Then $R/N$ is pronilpotent, which means $G/N$ is virtually pronilpotent.

As $G$ itself is not virtually pronilpotent, it follows that $N$ has finite index in $G$, so there is some $K/L \in \mc{D}$ with $K \le N$.  As $N$ centralizes all such chief factors, $K/L$ is abelian, say of exponent $p$, while all nonabelian simple composition factors of $N$ have Schur multipliers of order coprime to $p$.  Thus, $\OO^p(N) < N$ by Lemma~\ref{opsch}.  As $\OO^p(N)$ is characteristic in $N$, it is normal in $G$.  But then $R/\OO^p(N)$ is an image of $R$ with a $G$-chief series whose factors are all in $\mc{D}$, contradicting the definition of $N$.\end{proof}

\begin{cor}\label{mainjithm:centralizers}Let $G$ be a just infinite profinite group that is not virtually pronilpotent, and let $\mc{C}_1,\mc{C}_2,\dots$ be a sequence of classes of finite groups, all accounting for infinitely many chief factors of $G$.  Suppose each class $\mc{C}_n$ also satisfies $\condb$.  Then $G$ is the limit of an inverse system of the form specified in Theorem~\ref{mainjithm}, with the additional condition that $\CC_{G_n}(P_n) < A_n$ for all $n$.\end{cor}

\begin{proof}
In the construction of the inverse system in Theorem~\ref{mainjithm}, we were free to choose any chief factor $R/S$ such that $R \le L$ and $R/S \in \mc{C}_{n+1}$.    It follows immediately from Corollary~\ref{chiefin} and Lemma~\ref{seclem} that there will be a choice for which $\CC_G(R/S) \le L$ and hence $\CC_G(K/\M_G(K)) \le L$, where $K$ is the associated narrow subgroup, which implies $\CC_{G_n}(P_n) < A_n$ in the inverse limit construction.
\end{proof}

The need for a condition like $\condb$ becomes clear when one considers an iterated transitive wreath product $G$ of copies of some finite perfect group $P$ with nontrivial centre.  Although $G$ is just infinite, has infinitely many abelian chief factors and is not virtually pronilpotent, all abelian chief factors of $G$ are central factors.

\section{Characterizations of the hereditarily just infinite property}

We can adapt Theorem~\ref{mainjithm} (using an idea from \cite{Wil}) to obtain a universal construction for hereditarily just infinite groups.

\begin{defn}A finite group $H$ is a \emph{central product} of subgroups $\{H_i \mid i \in I\}$ if these subgroups generate $H$, and whenever $i \not=j$ then $[H_i,H_j]=1$.  If $H$ is a normal subgroup of a group $G$, we say $H$ is a \defbold{basal central product} in $G$ if in addition the groups $\{H_i \mid i \in I\}$ form a conjugacy class of subgroups in $G$.

Say $H$ is \emph{centrally indecomposable} if it cannot be expressed as a central product of proper subgroups, and \emph{basally centrally indecomposable in $G$} if it cannot be expressed as a basal central product in $G$.\end{defn}

\begin{thm}\label{mainjithm:hji}Let $G$ be a hereditarily just infinite profinite group.  Let $\mc{C}_1,\mc{C}_2,\dots$ be a sequence of classes of finite groups, such that $G$ has infinitely many chief factors in $\mc{C}_n$ for all $n$.  Then $G$ is the limit of an inverse system $\Lambda = \{(G_n)_{n > 0},\rho_n: G_{n+1} \twoheadrightarrow G_n\}$ as follows:

Each $G_n$ has a specified nontrivial normal subgroup $A_n$ such that, setting $P_n = \rho_{n}(A_{n+1})$:
\begin{enumerate}[(i)]  
\item $\M_{G_{n+1}}(A_{n+1}) \ge \ker\rho_{n} \ge P_{n+1}$;
\item Each normal subgroup of $G_n$ contains $P_n$ or is contained in $A_n$ (or both);
\item $P_n$ is a minimal normal subgroup of $G_n$;
\item $P_n \in \mc{C}_n$ for all $n$;
\item In $G_n$, every normal subgroup containing $A_n$ is basally centrally indecomposable.
\end{enumerate}

Conversely, any inverse system satisfying conditions (i) and (ii) for all but finitely many $n$ and (v) for infinitely many $n$ has a limit that is hereditarily just infinite.\end{thm}

\begin{proof}Suppose that $G$ is hereditarily just infinite with the specified chief factors.  We will obtain an infinite descending chain $(K_n)_{n \ge 0}$ of narrow subgroups of $G$ as in the proof of Theorem~\ref{mainjithm}, additionally ensuring that (v) will be satisfied.

Suppose that $G$ is not virtually abelian; note that in this case, $|H:[H,H]|$ is finite for every nontrivial normal subgroup $H$ of $G$.  Let $K_0 = G$.  Suppose $K_n$ has been chosen for some $n \ge 0$; let $M = \Ob^*_G(\M(K_n))$, and let $L$ be a normal subgroup of $G$ properly contained in $[M,M]$.  As in the proof of Theorem~\ref{mainjithm}, there is a chief factor $R/S$ of $G$ such that $R \le L$ and $R/S \in \mc{C}_n$, and $K_{n+1} \nar G$ associated to $R/S$.  Then $K_{n+1} \le L$ and $K_{n+1}/\M_G(K_{n+1}) \cong R/S$.  Set $G_n = G/\M_G(K_{n+1})$, set $P_n = K_{n+1}/\M_G(K_{n+1})$, set $A_n = K_{n}/\M_G(K_{n+1})$ and let the maps $\rho_n$ be the natural quotient maps.    

If instead $G$ is virtually abelian, then it is easy to see that in fact $G$ has a unique largest abelian open normal subgroup $N$, such that $\CC_G(N) = N \cong \bZ_p$ for some prime $p$.  Note that $\mc{C}_n$ must contain the cyclic group of order $p$.  Let $K_0 = N$ and suppose $K_n$ has been chosen for some $n \ge 0$.  Then there are open subgroups $L_n$ and $K_{n+1}$ of $K_n$ such that for all $g \in G \setminus N$ we have
\[
[\Ob_G(\M(K_n)),\langle g \rangle] \ge L_n \text{ and } [L_n,\langle g \rangle] \ge K_{n+1}.
\]
(If $G = N$, we simply take $K_{n+1} = \M(K_n)$.)  We then define $G_n$, $P_n$, $A_n$ and $\rho_n$ as before.

In both cases, the inverse limit of $(G_n,\rho_n)$ is an infinite quotient of $G$, which is in fact equal to $G$ since $G$ is just infinite.  Conditions (i)--(iv) are satisfied as in the proof of Theorem~\ref{mainjithm}.  Let $\{ \pi_n: G \rightarrow G_n \mid n > 0\}$ be the surjections associated to the inverse limit.  We now prove by contradiction that (v) is satisfied.  Let $U$ be a normal subgroup of $G_n$ containing $A_n$, and suppose $U$ is a proper basal central product in $G_n$: say $U$ is the normal closure of $V$, where $V < U$ and the distinct $G_n$-conjugates of $V$ centralize each other.  Note that $V$ is normal in $U$ and hence normalized by $A_n$; thus $\pi\inv_n(V)$ is normalized by $K_n$.

If $G$ is not virtually abelian, then $V$ is not contained in $\M(K_n)$; it follows that $\pi\inv_n(V) \ge \Ob^*_G(\M(K_n))$.  By the same argument, $\pi\inv_n(gVg\inv) \ge \Ob^*_G(\M(K_n))$ for every $g \in G$.  Thus all $G_n$-conjugates of $V$ contain the group $\pi_n(\Ob^*_G(\M(K_n)))$, which is nonabelian as a consequence of the choice of $L$.  If $G$ is virtually abelian, the existence of a nontrivial decomposition ensures that $U$ is not cyclic of prime power order, so $\pi\inv_n(U) \not\le N$.  Thus $\pi\inv_n(U)$ contains $K_n$ and $\pi\inv_n(V)$ contains some $g \in G \setminus N$.  We then have $\pi\inv_n(V) \ge [K_n,\langle g \rangle] \ge L_n$; indeed $\pi\inv_n(hVh\inv) \ge L_n$ for all $h \in G_n$.  The choice of $L_n$ ensures that $L_n/\M_G(K_{n+1})$ is a noncentral subgroup of $hVh\inv$ for all $h \in G_n$.  In either case we obtain a subgroup of $\bigcap_{h \in G_n}hVh\inv$ that is not centralized by $V$, contradicting the hypothesis that the conjugates of $V$ form a basal central decomposition of $U$.  Thus (v) holds for all $n > 0$.

\

Conversely, suppose we are given an inverse system satisfying (i) and (ii) for all $n \ge n_0>1$, and also (v) for infinitely many $n$.  Let $G$ be the inverse limit and let $\pi_n: G \rightarrow G_n$ be the surjections associated to the inverse limit.  Then $G$ is just infinite by Theorem~\ref{mainjithm}.  As in the proof of Theorem~\ref{mainjithm}, the groups $\pi\inv(A_n)$ for $n \ge n_0$ form a descending chain of open normal subgroups of $G$ with trivial intersection.  If $G$ is not hereditarily just infinite, then it would have a nontrivial subgroup $V$ of infinite index such that the distinct conjugates of $V$ centralize each other: this is clear if $G$ is virtually abelian, and given by \cite[2.1]{Wil} if $G$ is not virtually abelian.  Let $U$ be the normal closure of $V$ in $G$.  Then for all but finitely many $n$, we see that $\pi\inv_n(A_n) \le U$ and $\pi_n(V) < \pi_n(U)$, so there exists $n$ such that (v) holds, $A_n \le \pi_n(U)$ and $\pi_n(V) < \pi_n(U)$.  But then $\pi_n(U)$ is the normal closure of $\pi_n(V)$ in $G_n$, and the distinct conjugates of $\pi_n(V)$ centralize each other, contradicting (v).  From this contradiction, we conclude that $G$ is hereditarily just infinite.
\end{proof}

We derive a construction from \cite{Wil} as a special case.  (Actually the conclusion is slightly stronger than stated in \cite{Wil}: the inverse limit is hereditarily just infinite, even if it is virtually abelian.)

\begin{thm}[{See Wilson \cite[2.2]{Wil}}]\label{wilthm} Let $G$ be the inverse limit of a sequence $(H_n)_{n \geq 0}$ of finite groups and surjective homomorphisms $\phi_n: H_n \twoheadrightarrow H_{n-1}$.  For each $n \geq 1$ write $K_n = \ker\phi_n$, and suppose that for all $L \unlhd H_n$ such that $L \nleq K_n$ the following assertions hold:
\begin{enumerate}[(i)]  
\item $K_n < L$;
\item $L$ has no proper subgroup whose distinct $H_n$-conjugates centralize each other and generate $L$.
\end{enumerate}
Then $G$ is a hereditarily just infinite profinite group.\end{thm}

\begin{proof}
After relabelling, we may assume that none of the homomorphisms $H_n \twoheadrightarrow H_{n-1}$ are injective.  Set
\[
G_n = H_{2n+2}; \; P_n = \ker\phi_{2n+2}; \; Q_n = \ker(\phi_{2n+1}\phi_{2n+2}); \; A_n = \ker(\phi_{2n}\phi_{2n+1}\phi_{2n+2}).
\]
Notice that $\rho_n := \phi_{2n+3}\phi_{2n+4}$ is a surjective homomorphism from $G_{n+1}$ to $G_n$ with kernel $Q_{n+1}$.  We see that every normal subgroup $U$ of $G_n$ containing $A_n$ properly contains $P_n$, and hence by (ii), $U$ is basally centrally indecomposable in $G_n$.  Thus condition (v) of Theorem~\ref{mainjithm:hji} is satisfied.  Let $M$ be a maximal $G_{n+1}$-invariant subgroup of $A_{n+1}$.  Then $M$ is not properly contained in $Q_{n+1}$; by considering the quotient $H_{2n+3}$ of $G_{n+1}$ and applying (i), we see that $M \ge Q_{n+1}$.  Thus $M_{G_{n+1}}(A_{n+1}) \ge \ker\rho_n$; clearly also $\ker\rho_n \ge P_{n+1}$.  Thus condition (i) of Theorem~\ref{mainjithm:hji} is satisfied.  Condition (ii) of Theorem~\ref{mainjithm:hji} also follows from condition (i) of the present theorem.  Hence by Theorem~\ref{mainjithm:hji}, $G$ is hereditarily just infinite.
\end{proof}

For virtually pronilpotent groups, there is an alternative way to specialize Theorem~\ref{mainjithm} to a characterization of the hereditarily just infinite property, without having to check that a collection of normal subgroups are basally centrally indecomposable.  We appeal to some results from \cite{ReiFJ}.

\begin{lem}\label{fijilem}
Let $G$ be a profinite group with no nontrivial finite normal subgroups. 
\begin{enumerate}[(i)]  
\item (\cite[Lemma~4]{ReiFJ}) If $G$ has a (hereditarily) just infinite open subgroup, then $G$ is (hereditarily) just infinite.
\item (\cite[Theorem~2]{ReiFJ}) If $G$ is pro-$p$ and $\Phi(G)$ is just infinite, then $G$ is hereditarily just infinite.\end{enumerate}
\end{lem}

This leads naturally to our modified construction.

\begin{thm}\label{hjivp}
Let $\Lambda = \{(G_n)_{n > 0},\rho_n: G_{n+1} \twoheadrightarrow G_n\}$ be an inverse system of finite groups.  Let $F_n = \Phi(\OO_p(G_n))$, let $\triv < A_n \unlhd F_n$ and let $P_{n} = \rho_{n}(A_{n+1})$ for $n > 0$.  Suppose that the following conditions hold for all $n > 0$:
\begin{enumerate}[(i)]
\item $\M_{F_{n+1}}(A_{n+1}) \ge \ker\rho_n \ge P_{n+1}$;
\item Each normal subgroup of $F_n$ contains $P_n$ or is contained in $A_n$ (or both);
\item $F_{n+1} = \rho\inv_n(F_n)$;
\item Every minimal normal subgroup of $G_1$ is contained in $F_1$.
\end{enumerate}
Then $G$ is hereditarily just infinite and virtually pro-$p$.  Conversely, every hereditarily just infinite virtually pro-$p$ group is the limit of such an inverse system.
\end{thm}

\begin{proof}
Suppose $\Lambda$ is an inverse system as specified, with inverse limit $G$.  Let $F = \Phi(\OO_p(G))$.  By condition (iv), $F = \varprojlim F_n$ and $|G:F| = |G_1:F_1| < \infty$, so $F$ is an open pro-$p$ subgroup of $G$.  Moreover, $F$ is just infinite by Theorem~\ref{mainjithm}.  Let $K$ be a nontrivial normal subgroup of $G$.  If $F \cap K = \triv$, then $\pi_1(FK) = F_1 \times \pi_1(K)$ where $\pi_1: G \rightarrow G_1$ is the natural projection.  Since $\ker\pi_1 \le F$, the group $\pi_1(K)$ is nontrivial and we have a contradiction to (iv).  Thus every nontrivial normal subgroup of $G$ intersects $F$ nontrivially, and is hence infinite; by Lemma~\ref{fijilem} it follows that $G$ is hereditarily just infinite.

Now let $G$ be a hereditarily just infinite virtually pro-$p$ group.  Let $F = \Phi(\OO_p(G))$.  Then $F$ is open in $G$ and in particular just infinite.  The construction of a suitable inverse system is analogous to that in the proof of Theorem~\ref{mainjithm}.

Choose $R_1 \nar F$ such that $R_1 \le \M(\Ob^*_G(F))$ and $S_1$ to be the core of $\M_F(R_1)$ in $G$.  Thereafter, choose $R_{n+1} \nar F$ such that $R_{n+1} \le \Ob_F(\M_F(S_n))$ and set $S_{n+1}$ to be the core of $\M_F(R_{n+1})$ in $G$.

Set $G_n = G/S_{n+1}$, $A_n = R_n/S_{n+1}$ and $P_n = R_{n+1}/S_{n+1}$.  Then conditions (i) and (ii) are satisfied for the same reasons as in Theorem~\ref{mainjithm}.  The choice of $S_1$ ensures conditions (iii) and (iv).
\end{proof}

\section{Further remarks on nonabelian chief factors}

Given a narrow normal subgroup $A$ of $G$, there are some alternative descriptions of $\M_G(A)$ in the case that $A/\M_G(A)$ is either central in $G$ or topologically perfect.

\begin{prop}\label{nonabmel}Let $G$ be a profinite group and let $1 < A \unlhd G$.
\begin{enumerate}[(i)]
\item We have $[A,\M_G(A)] \le \M(A)$.  If $A/\M_G(A)$ is topologically perfect, then $\M_G(A) = \M(A)$.
\item Suppose $A > \M(A)[A,G]$.  Then $A \nar G$ if and only if $|A:\M(A)[A,G]|$ is prime and every chief factor of $G$ occurring as a quotient of $A$ is central.  If $A \nar G$ then $\M_G(A) = \M(A)[A,G]$.
\item If $G$ is pronilpotent, then $\M_G(A) = \M(A)[A,G]$.
\end{enumerate}
\end{prop}

\begin{proof}
(i)
Let $K$ be a normal subgroup of $A$ such that $A/K$ is a nonabelian simple group and let $L$ be the core of $K$ in $G$.  Then $A/L$ is isomorphic to a subdirect product of copies of $A/K$; since $A/K$ is a nonabelian simple group, in fact $A/L$ is isomorphic to a direct product of copies of $A/K$, say $A/L = F_1 \times F_2 \times \dots \times F_n$ where $F_i \cong A/K$.  We see that $K/L$ contains all but one of the simple direct factors of $A/L$, say $K/L = F_2 \times \dots \times F_n$.  Suppose $\M_G(A) \nleq L$; then there exists $M < A$ such that $M$ is $G$-invariant and $L < M$.  Then either $F_1 \nleq M/L$, in which case $M \le K$, or else $F_i \nleq M/L$ for $i > 1$, in which case $\CC_{G/L}(M/L)$ is a nontrivial $G$-invariant subgroup of $K/L$.  In either case we have a contradiction to the fact that $L$ is the core of $K$ in $G$.  Thus in fact $\M_G(A) \le L$ and in particular $\M_G(A) \le K$, showing that $A/\M_G(A)$ accounts for all nonabelian simple quotients of $A$.  The group $A/\M(A)$ is by construction a subdirect product of finite simple groups.  From there, it is easily seen that in fact $A/\M(A) = \Z(A/\M(A)) \times P/\M(A)$ where $P/\M(A)$ is perfect; indeed $P/\M(A)$ is a direct product of finite simple groups.  Since $A/\M_G(A)$ accounts for all nonabelian finite simple quotients, we must have $\M_G(A) \cap P \le \M(A)$, so $\M_G(A)/\M(A)$ is central in $A/\M(A)$ as required.  If $A/\M_G(A)$ is topologically perfect, then there are no $G$-invariant subgroups $N$ of $A$ such that $A/N$ is an abelian chief factor of $G$.  Since $P$ is characteristic in $A$ and $A/P$ is abelian, it follows that $P = A$ and hence $\M_G(A) = \M(A)$.

(ii)
Let $B = \M(A)[A,G]$.  Then $A/B$ is central in $G/B$ and $\M(A/B)=1$, so $A/B$ is residually an abelian simple group.

If $|A:B|$ is not prime, then $A/B$ is of the form $K_1/B \times K_2/B$ where $K_1/B$ is of prime order and $K_2/B$ is nontrivial.  Thus $K_1$ and $K_2$ are proper subgroups of $A$ normalized by $G$ and $K_1K_2 = A$.  If instead there is a chief factor $A/C$ of $G$ that is not central, then $[A,G] \not< C$, so $BC = A$, but both $B$ and $C$ are proper subgroups of $A$.  In either case $A$ is not narrow in $G$.

Now suppose $|A:B|=p$ and that any chief factor of $G$ of the form $A/C$ is central.  Let $K < A$ such that $A/K$ is a chief factor of $G$.  Then $A/K$ is a direct product of simple groups, so $K \ge \M(A)$, and $A/K \le \Z(G/K)$, so $K \ge [A,G]$ and hence $K \ge B$.  Since $|A:B|$ is prime, in fact we must have $K=B$.  Thus every proper $G$-invariant subgroup of $A$ is contained in $B$, in other words $A \nar G$ with $\M_G(A)=B$.

(iii)
Suppose $G$ is pronilpotent.  Then every chief factor of $G$ is central, so $\M_G(A) \ge [A,G]$; clearly also $\M_G(A) \ge \M(A)$, so $\M_G(A) \ge \M(A)[A,G]$.  On the other hand, the quotient $V = A/\M(A)[A,G]$ is a central factor of $G$, such that every Sylow subgroup of $V$ is elementary abelian; it then follows easily that for every nontrivial element $v$ of $V$, there is some maximal subgroup $W$ of $V$ that does not contain $v$.  Thus the $G$-invariant quotients of $A$ of prime order separate elements of $V$, so $\M_G(A) = \M(A)[A,G]$.
\end{proof}

We now give an inverse system characterization of a class of hereditarily just infinite groups that are not virtually prosoluble.  The class obtained is less general than that given by Theorem~\ref{mainjithm:hji}, but the conditions here are easier to check and still provide some interesting examples.  Note that the conditions in brackets are automatic, by the positive answer to the Schreier conjecture; Theorem~\ref{primhji} has been stated in such a way to avoid using any deep results about finite simple groups.  Proposition~\ref{intro:subprim} will follow immediately given the positive answer to the Schreier conjecture.

\begin{thm}\label{primhji}Let $\Lambda = \{(G_n)_{n > 0},\rho_n: G_{n+1} \twoheadrightarrow G_n\}$ be an inverse system of finite groups.  Let $A_n$ be a normal subgroup of $G_n$ and set $P_{n} = \rho_n(A_{n+1})$ for $n > 0$.  Suppose the following conditions hold for sufficiently large $n$:
\begin{enumerate}[(i)]  
\item $\M(A_{n+1}) = \ker\rho_{n} \ge P_{n+1}$;
\item $A_n > \CC_{G_n}(P_n)$;
\item $P_n$ is a direct product of nonabelian simple groups, such that $G_n$ permutes the simple factors of $P_n$ subprimitively by conjugation (and $\N_{G_n}(F)/F\CC_{G_n}(F)$ is soluble for each simple factor $F$ of $P_n$).\end{enumerate}

Then $G = \varprojlim G_n$ is hereditarily just infinite and not virtually prosoluble.  Moreover, this construction accounts for all profinite groups $G$ for which the following is true:

$\condc$ $G$ is just infinite, and there are infinitely many nonabelian chief factors $R/S$ of $G$ such that the simple factors of $R/S$ are permuted subprimitively by conjugation in $G/S$ (and such that $\N_{G/S}(F)/F\CC_{G/S}(F)$ is soluble for each simple factor $F$ of $R/S$).\end{thm}

For the proof, we note a characterization of subprimitive actions.

\begin{lem}\label{lem:subprim}
Let $G$ be a group acting on a set $X$.  Then the following are equivalent:
\begin{enumerate}[(i)]
\item $G$ acts subprimitively on $X$;
\item For every $L \unlhd K \unlhd G$, either $L$ acts trivially on $X$, or $L$ fixes no point in $X$.
\end{enumerate}
\end{lem}

\begin{proof}
Let $N$ be the kernel of the action of $G$ on $X$.

Suppose $G$ acts subprimitively on $X$; let $L \unlhd K \unlhd G$, and let $Y$ be the set of fixed points of $L$.  Then $Y$ is a $K$-invariant set; thus if $Y$ is nonempty, then $K/(K \cap N)$ acts faithfully on $Y$.  Since $L$ acts trivially on $Y$, we conclude that $L \le K \cap N$, so $L$ acts trivially on $X$.  Thus (i) implies (ii).

Conversely, suppose that $G$ does not act subprimitively on $X$.  Then there is a normal subgroup $K$ and a $K$-orbit $Y$, such that $K/(K \cap N)$ does not act faithfully on $Y$.  The fixator $L$ of $Y$ in $K$ is then a subgroup of $G$ that acts nontrivially on $X$ (since $K \ge L > K \cap N$), has a fixed point in $X$ and satisfies $L \unlhd K \unlhd G$.  Thus (ii) implies (i).
\end{proof}

\begin{proof}[Proof of Theorem~\ref{primhji}]Suppose $\Lambda$ is an inverse system as specified, with inverse limit $G$; let $\pi_n: G \rightarrow G_n$ be the associated projection map.  By renumbering, we may assume that the given conditions hold for all $n$.  Note that conditions (i) and (iii) imply that $A_n$ is perfect and $\M(A_n) = \M_{G_n}(A_n)$ for all $n > 1$.  Condition (ii) ensures that $P_n$ is nontrivial; conditions (i) and (iii) then ensure that $G$ has a composition series with infinitely many nonabelian factors, that is, $G$ is infinite and not virtually prosoluble.

We claim now that $G$ is hereditarily just infinite.  To prove this claim, it suffices to show (given Lemma~\ref{fijilem}) that every open normal subgroup of $G$ is just infinite, that is, given any nontrivial subgroup $H$ of $G$ such that $H \unlhd K \unlhd_o G$ for some $K$, then $H$ is open in $G$.  Letting $H_n = \pi_n(H)$, there exists $n_0$ such that $H_n$ is not contained in $A_n$, for all $n \ge n_0$, so by condition (i), $H_n$ does not centralize $P_n$.  This implies $H_n \cap P_n$ is a nontrivial subnormal subgroup of $G_n$, which is therefore insoluble, so $H$ cannot be prosoluble.  Thus, the intersection $I$ of the terms in the derived series of $H$ is nontrivial.  Clearly $I$ is also normal in $K$, so without loss of generality we may assume $H=I$, that is, $H$ is perfect.  We can now draw a stronger conclusion about $H_n$ for sufficiently large $n$: since $H_n$ is perfect and not contained in $P_n\CC_{G_n}(P_n)$, the quotient $H_n/(H_n \cap P_n\CC_{G_n}(P_n))$ is insoluble.  This ensures that $H_n$ acts nontrivially on the set $\Omega_n$ of simple direct factors of $P_n$.  Since $H_n \unlhd \pi_n(K) \unlhd G_n$, it follows by Lemma~\ref{lem:subprim} that $H_n$ acts with no fixed points on $\Omega_n$ and hence $[H_n,P_n] = P_n$.  In the case when $H$ is normal in $G$, we conclude that $P_n \le H_n$ for $n$ sufficiently large.  In particular, we must have $P_n \le \pi_n(K)$ for $n$ sufficiently large.  But then the fact that $H_n$ is normal in $\pi_n(K)$ ensures that $P_n$ normalizes $H_n$, so $P_n \le H_n$.  It then follows that $A_{n+1} \le H_{n+1}\ker\rho_n = H_{n+1}\M_{G_{n+1}}(A_{n+1})$, so $A_{n+1} \le H_{n+1}$ by Lemma~\ref{critlem}; in particular, $\ker\rho_n \le H_{n+1}$ for $n$ sufficiently large, ensuring that $H$ is open in $G$ as desired.

Conversely, let $G$ be a profinite group satisfying $\condc$.  As in previous proofs, we construct a descending chain $(K_n)$ of open normal subgroups of $G$.

Let $K_0 = G$.  Suppose $K_n$ has been chosen.  Then $\M(K_n) \leq_o G$ by Lemma \ref{melfin} so $\Ob_G(\M(K_n)) \le_o G$ by Theorem \ref{genob}.  Given Lemma~\ref{narrowassoc}, we can find $K_{n+1} \nar G$ such that $K_{n+1} \le \Ob_G(\M(K_n))$, $K_{n+1}/\M_G(K_{n+1})$ is perfect, the simple factors of $K_{n+1}/\M_G(K_{n+1})$ are permuted subprimitively by conjugation in $G/\M_G(K_{n+1})$, and the group of outer automorphisms induced by $G$ on each simple factor is soluble; additionally $\M_G(K_{n+1})=\M(K_{n+1})$ by Proposition \ref{nonabmel}.  Moreover, we have infinitely many choices for $K_{n+1}$; suppose that for all but finitely many of them, $\CC_G(K_{n+1}/\M(K_{n+1})) \not\le \M(K_n)$.  Then by Theorem \ref{genob}, $D = \bigcap_K \CC_G(K/\M(K))$ would have finite index, where $K$ ranges over all possible choices for $K_{n+1}$.  By Theorem~\ref{genob} again, this would imply in turn that $K\le D$ for all but finitely many choices of $K$, which is absurd given that $K/\M(K)$ is perfect and $D$ centralizes $K/\M(K)$.  We can therefore ensure $\CC_G(K_{n+1}/\M(K_{n+1})) \le \M(K_n)$.  Now set $G_n = G/\M(K_{n+1})$, set $A_n = K_n/\M(K_{n+1})$, set $P_n = K_{n+1}/\M(K_{n+1})$ and let the maps $\rho_n$ be the natural quotient maps $G_{n+1} \rightarrow G_n$.  This produces an inverse system for $G$ with all the required properties.
\end{proof}

\begin{rem}A just infinite branch group $G$ can certainly have infinitely many nonabelian chief factors.  However, given condition $\condc$, the permutation action of $G$ on the simple factors of a chief factor can only be subprimitive for finitely many of these chief factors.\end{rem}

The following shows the flexibility of the conditions in Theorem~\ref{primhji}.

\begin{ex}\label{primex}Let $X_0$ be a finite set with at least two elements, let $G_{-1}$ be a nontrivial subprimitive subgroup of $\mathrm{Sym}(X_0)$, let $S_0,S_1,\dots$ be a sequence of nonabelian finite simple groups and let $F_0,F_1,\dots$ be a sequence of nontrivial finite perfect groups.  Set $G_0 = S_0 \wr_{X_0} G_{-1}$, that is, the wreath product of $S_0$ with $G_{-1}$ where the wreathing action is given by the natural action of $G_{-1}$ on $X_0$.  Thereafter $H_n$ is constructed from $G_n$ as $H_n = F_n \wr_{Y_n} G_n$, where the action of $G_n$ on $Y_n$ is subprimitive and faithful, and $G_{n+1}$ is constructed from $H_{n}$ as $G_{n+1} = S_{n+1} \wr_{X_{n+1}} H_n$ where the action of $H_n$ on $X_{n+1}$ is subprimitive and faithful.  Set $\rho_n:G_{n+1} \rightarrow G_n$ to be the natural quotient map from $S_{n+1} \wr_{X_{n+1}} (F_n \wr_{Y_n} G_n)$ to $G_n$ for $n \ge 0$, and let $G$ be the inverse limit arising from these homomorphisms.  For $n \ge 0$ set $A_{n+1}$ to be the kernel of the natural projection of $G_{n+1}$ onto $H_{n-1}$; thus $A_{n+1}$ is a direct product of $|X_{n}|$ copies of $S_{n+1} \wr_{X_{n+1}} (F_n \wr_{Y_{n}} S_n)$.  We observe that $A_n$ is perfect and $\M(A_n) = \ker\rho_{n-1}$.  Moreover, $P_n$ is a normal subgroup of $G$ isomorphic to a power of $S_n$, with $G_n$ permuting the copies of $S_n$ subprimitively and $\CC_{G_n}(P_n) = \triv$.  Thus $G$ is hereditarily just infinite by Theorem \ref{primhji}.

Note the following:
\begin{enumerate}[(i)]  
\item Every nontrivial finite group can occur as the initial permutation group $G_{-1}$.  There is also a great deal of freedom in the choice of the groups $F_n$, as there are general methods to embed a finite group as a `large' subnormal factor of a finite perfect group.  For instance if $B$ is a finite group, then $B \wr A_5$ has a perfect normal subgroup of the form $K \rtimes A_5$ where $K = \{f: \{1,2,\dots,5\} \rightarrow B \mid \prod^5_{i=1}f(i) \in [B,B]\}$; clearly $B$ appears as a quotient of $K$.  The groups $F_n$ could thus be chosen so that in the resulting inverse limit $G$, every finite group appears as a subnormal factor $K/L$, such that $L \unlhd K \unlhd H \unlhd G$.
\item Provided the $S_n$ or $F_n$ are chosen to be `universal' (that is, so that for every finite group $F$ there are infinitely many $n$ such that $F$ embeds into $S_n$ or $F_n$), then every countably based profinite group embeds into every open subgroup of $G$, because $G$ contains a closed subgroup $\prod_n S_n \times \prod_n F_n$ (take the `diagonal' subgroup at each level of the wreath product).  This gives an alternative proof of Theorem A of \cite{Wil}.
\item The group constructed by Lucchini in \cite{Lucchini} is a special case of the construction under discussion, and indeed an example of the previous observation.  This is especially interesting as Lucchini's paper predates \cite{Wil}, and appears to be the earliest example in the literature of a hereditarily just infinite profinite group that is not virtually pro-$p$.  (It is proven in \cite{Lucchini} that the group is just infinite, but the hereditarily just infinite property is not considered.)
\item There are exactly $2^{\aleph_0}$ commensurability classes of groups in this family of examples.  Consider for instance the set $\mc{S}$ of nonabelian finite simple groups occurring infinitely many times in a composition series for $G$.  Then $\mc{S}$ only depends on $G$ up to commensurability, but can be chosen to be an arbitrary nonempty subset of the set of all isomorphism classes of nonabelian finite simple groups.  On the other hand there cannot be more than $2^{\aleph_0}$ examples, since every just infinite profinite group is countably based and there are only $2^{\aleph_0}$ countably based profinite groups up to isomorphism.
\item There are interesting infinite ascending chains inside this family of examples.  For instance, let $T_k$ be the group formed in the inverse limit if one replaces $G_k$ with $S_k \wr_{X_k} \mathrm{Sym}(X_{k})$ in the construction and extends from there as before.  Then $G$ embeds naturally into $T_k$ as an open subgroup, and if $k' > k$ then $T_k$ embeds naturally into $T_{k'}$ as an open subgroup.  Taking the direct limit of the groups $T_k$, one obtains a totally disconnected, locally compact group $T$.  This group has a simple locally finite abstract subgroup $A$, formed as the union of the groups $\mathrm{Alt}(X_{k})$, where $\mathrm{Alt}(X_{k}) \le \mathrm{Sym}(X_{k}) \le T_k$.  The smallest closed normal subgroup $N$ of $T$ containing $A$ will then be a topologically simple open subgroup of $T$.  We thus obtain a totally disconnected, locally compact, topologically simple group $N$ such that every closed subgroup with open normalizer is open, but also (for a suitable choice of $S_n$ or $F_n$ as in (ii)) such that every identity neighbourhood in $N$ contains a copy of every countably based profinite group as a closed subgroup.
\end{enumerate}
\end{ex}

\end{document}